\documentclass[12pt]{article}
\usepackage{amsfonts,amsthm,latexsym,amsmath,amssymb, enumerate, amsmath, amsfonts, amscd, amsthm}
\usepackage{caption}

\newtheorem{theo+}           {Theorem}
\newtheorem{prop+}           {Proposition}
\newtheorem{coro+}           {Corollary}
\newtheorem{lemm+}           {Lemma}

\newtheorem{ques+}{Question}
\theoremstyle{definition}

\newtheorem{not+}            {Notation}

\newtheorem{Ex}              {Example}[section]

\newtheorem*{ack}            {Acknowledgements}

\newtheorem{rema+}           {Remark}

\newenvironment{theorem}{\begin{theo+}}{\end{theo+}}
\newenvironment{proposition}{\begin{prop+}}{\end{prop+}}
\newenvironment{corollary}{\begin{coro+}}{\end{coro+}}
\newenvironment{lemma}{\begin{lemm+}}{\end{lemm+}}
\newenvironment{remark}{\begin{rema+}}{\end{rema+}}

\newenvironment{question}{\begin{ques+}}{\end{ques+}}

\newcommand{\bq}{\begin{eqnarray}}
\newcommand{\eq}{\end{eqnarray}}
\newcommand{\beq}{\begin{eqnarray*}}
\newcommand{\eeq}{\end{eqnarray*}}

\date{}

\begin{document}

\thispagestyle{empty}
\title{Weakly-morphic modules}
\author{Philly Ivan  Kimuli\footnote{Corresponding author}~   and David Ssevviiri}
\maketitle
\begin{center}
Department of Mathematics, Makerere University,\\
P.O. Box 7062, Kampala Uganda\\
E-mail addresses:  kpikimuli@gmail.com; david.ssevviiri@mak.ac.ug
\end{center}
\begin{abstract}
\noindent Let $R$ be a commutative ring, $M$ an $R$-module and $\varphi_a$ be  the endomorphism  of $M$ given by  right multiplication by
$a\in R$.    We say that $M$ is {\it weakly-morphic} if $M/\varphi_a(M)\cong \ker(\varphi_a)$ as $R$-modules for every   $a$.  We study these  modules  and use them to  characterise the rings $R/\text{Ann}_R(M)$, where $\text{Ann}_R(M)$ is the right annihilator of $M$.   A kernel-direct or image-direct module $M$ is weakly-morphic if and only if  each element of $R/\text{Ann}_R(M)$ is regular as an endomorphism element  of $M$.    If $M$ is a weakly-morphic module over an integral domain $R$, then $M$ is torsion-free if and only if it is divisible if and only if $R/\text{Ann}_R(M)$ is a field.  A finitely generated  $\Bbb Z$-module is  weakly-morphic if and only if it is finite; and it is morphic if and only if it is weakly-morphic  and each of its primary  components is of the form  $(\Bbb Z_{p^k})^n$ for  some non-negative integers $n$ and  $k$.
\end{abstract}



\maketitle

\maketitle
\section{Introduction}
Throughout this paper, all rings are  commutative with a non-zero identity and all modules are  unital right $R$-modules.
Let $R$ be a commutative ring with non-zero identity and  $M$ be a  non-zero unital right $R$-module.   $S:=\text{End}_R(M)$  denotes the ring of endomorphisms of $M$;  and $S_M\cong R/\text{Ann}_R(M)$ denotes the ring of endomorphisms of $M$   given by  right multiplication by $a\in R$ (where  $\text{Ann}_R(M)=\{r\in R:mr=0~\text{for every}~m\in M\}$ the right annihilator of $M$).   For $\varphi\in S, \ker(\varphi)$  denotes the kernel  of $\varphi$.
Following \cite{nicholson2005morphic},  $\varphi\in S$  is  {\it morphic} if $M/\varphi(M)\cong \ker(\varphi)$ as right $R$-modules.  The module $M$ is morphic if every $\varphi\in S$ is morphic.   If $M=R_R$, then this notion reduces to that of rings.
An element $a$ of a ring $R$ is called {\it morphic} if $R/aR\cong\text{Ann}_R(a)$, where $\text{Ann}_R(a)\text{:}=\{r\in R:ar=0\}$, as  $R$-modules. The ring $R$ is called  {\it  morphic} if every element of $R$ is  morphic.  The notion of morphic modules has its roots in the study of unit-regular rings.   G. Ehrlich introduced the class of unit-regular rings in \cite{ehrlich1968unit} and  \cite{ehrlich1976units} in which she proved that a  ring is unit-regular if and only if it is regular and morphic.  From 2003, the notion has been extensively studied and extended, for instance,  to  modules, groups,  rings and near-rings, see   \cite{bamunoba2020morphic, Calugareanu2010abelian, cualuguareanu2018abelian, Harmanci, li2010morphic,  nicholson2005survey, nicholson2004rings, nicholson2004principal, nicholson2005morphic} among others for details.\\

\noindent To any $R$-module $M$, we can associate the two rings:  $S$ and $S_M$.   Properties of the module $M$ can be characterised by either $S$ or $S_M$ and vice-versa.    For instance,   a module $M$ is: (i)  indecomposable if and only if $S$ does not contain any nontrivial idempotent elements \cite[Proposition 5.10]{anderson1992rings}; (ii)  morphic and either image-direct or kernel-direct if and only if  $S$ is unit-regular \cite[Theorem  35]{nicholson2005morphic}.    The main objective of this paper, besides studying properties of weakly-morphic modules, is to explore the relationship between weakly-morphic modules and the rings $R$ and $S_M$.   In particular,  we show that if  $M$ is a kernel-direct or image-direct module, then $M$ is weakly-morphic if and only if each element of $S_M$ is regular as an element of $S$ (see Theorem~\ref{datadata}); if $M$ is a weakly-morphic module over an integral domain, then $M$ is torsion-free if and only if it is divisible if and only if $S_M$ is a field (Theorem~\ref{fafa}); over a   principal ideal domain $D$, a finitely generated torsion-free $D$-module is weakly-morphic if and only if  $D$ is  a morphic ring (Proposition~\ref{capsa}).\\

\noindent It is further shown that a finitely generated Abelian  group is weakly-morphic as a $\Bbb Z$-module if and only if it is finite (Theorems~\ref{das}).  Hence, by \cite[Theorem 26]{nicholson2005morphic}, morphic finitely generated Abelian  groups are the  weakly-morphic $\Bbb Z$-modules whose $(p)$-primary  components are of the form  $(\Bbb Z_{p^k})^n$ for  some non-negative integers $n$  and $k$ (see Theorems~\ref{dsa}).
Finally, by Theorem~\ref{ftftf}, a Noetherian integral domain $D$ is Artinian if and only if every torsion-free $D$-module is weakly-morphic if and only if every divisible $D$-module is weakly-morphic.\\

\noindent By $\mathbb{Z}, \Bbb Z_n, \mathbb{Q}$ and $\mathbb{R}$ we denote the ring of integers, integers modulo $n$, rational and real numbers respectively.        The notation $N \subseteq M$(resp., $N\subseteq^{\bigoplus} M$) means that $N$ is a  submodule (resp., a direct summand) of $M$.
$\text{Ann}_{M}(a)\text{:}=\{m\in M:ma=0~\text{ for}~ a\in R\}$  and $\text{Ann}_{R}(m)\text{:}=\{r\in R:mr=0~\text{ for}~0\neq m\in M\}$.
 For each $a\in R$, $\varphi_a$  denotes the endomorphism  of $M$
given by  right multiplication by $a$.   The  image $\varphi_a(M)$ and kernel $\ker(\varphi_a)$ of $\varphi_a$, where necessary, are denoted by $Ma$ and  $\text{Ann}_{M}(a)$ respectively.   By $U(R)$ we mean  the group of units of $R$ and {\bf Mod-}$R$  denotes the category of all right $R$-modules.

\section{Basic properties of weakly-morphic modules}\label{sec2}
Let $R$ be a ring, $M$ an $R$-module and $\varphi_a$ be  the endomorphism  of $M$ given by  right multiplication by
$a\in R$.      We say that $M$ is $a$-{\it morphic}  if  $M/\varphi_a(M)\cong \ker(\varphi_a)$ as $R$-modules.   It
 is  a \emph{weakly-morphic module} if it is $a$-morphic for every $a\in R$, or equivalently if  $M/Ma\cong \text{Ann}_{M}(a)$  as $R$-modules  for every $a\in R$.
It is easy to see that for $M = R{}_{R}$, the notion of   weakly-morphic modules coincides with that of morphic rings; as the only endomorphisms on $R$ are those of the type $\varphi_a(r)= ar$.    Every morphic module is weakly-morphic but there exists weakly-morphic modules which are not morphic.

\begin{Ex}~\label{gal}
The $\mathbb{Z}$-modules $\mathbb{Z}_2$ and $\mathbb{Z}_4$ are morphic  but the direct sum  $M = \mathbb{Z}_2\bigoplus\mathbb{Z}_4$  is not morphic by \cite[Example 12]{nicholson2005morphic}.   However, here we claim that $M$ is a weakly-morphic $\mathbb{Z}$-module.  Let $r\in\mathbb{Z}$.  If $a=4r$, then  $Ma=0$ and $\text{Ann}_M(a)=M$.  If $a=4r+1$ or $4r+3$, then $Ma=M$ and $\text{Ann}_M(a)=0$.   In  both cases it is easy to see that $M/Ma\cong \text{Ann}_M(a)$.  The remaining case is $a=4r+2$ for which $Ma=0\bigoplus 2\mathbb{Z}_4$.   Again here $M/Ma=(\mathbb{Z}_2\bigoplus\mathbb{Z}_4)/(0\bigoplus 2\mathbb{Z}_4)\cong \mathbb{Z}_2\bigoplus \mathbb{Z}_2\cong\text{Ann}_{M}(a)$.  Thus $M$ is weakly-morphic.
\end{Ex}
\begin{lemma}\label{x}
Let  $M$ be an $R$-module and  $a\in R$.  The following statements are equivalent:
\begin{enumerate}[(1)]
\item[\emph{(1)}] $M$ is $a$-morphic.
\item[\emph{(2)}]  There exists $\psi\in S$ such that $Ma = \ker(\psi)$ and $\text{Ann}_{M}(a) = \psi(M).$
\item[\emph{(3)}] There exists $\psi\in S$ such that the sequence below is exact $$ {\cdots}\stackrel{\psi}{\longrightarrow}M\stackrel{\varphi_a}{\longrightarrow}M\stackrel{\psi}{\longrightarrow}M\stackrel{\varphi_a}{\longrightarrow}M\stackrel{\psi}{\longrightarrow}{\cdots}.$$
\end{enumerate}
\end{lemma}
\begin{proof}~
\begin{enumerate}
  \item [(1)$\Rightarrow$(2)] Given (1),  let $M \stackrel{\pi}{\longrightarrow} M/Ma\stackrel{\sigma}{\longrightarrow}\text{Ann}_{M}(a)$, where $\pi$ is the coset map and $\sigma$ is an isomorphism.  By setting $\psi:=\sigma\pi$, we get the required endomorphism.
  \item [(2)$\Rightarrow$(3)] By (2), there exists $\psi\in S$ such that $Ma=\ker(\psi)$ and $\psi(M)=\ker(\varphi_a)$.  It follows that the sequence ${\cdots}\stackrel{\psi}{\longrightarrow}M\stackrel{\varphi_a}{\longrightarrow}M\stackrel{\psi}{\longrightarrow}M\stackrel{\varphi_a}{\longrightarrow}M
      \stackrel{\psi}{\longrightarrow}{\cdots}$ is exact.
  \item [(3)$\Rightarrow$(1)]  By (3), $\text{Im}(\varphi_a)=\ker(\psi)$ and $\psi(M)=\ker(\varphi_a)$.  It follows that $M/\text{Im}(\varphi_a)=M/\ker(\psi)\cong \psi(M)=\ker(\varphi_a)$.  This proves (1).
\end{enumerate}
\end{proof}

\noindent Following the fact that $R_R$ is $a$-morphic if and only if $a\in R$ is a morphic element in $R$, we have Corollary~\ref{y} which compares with  \cite[Lemma 1]{nicholson2004rings}:

\begin{corollary}\label{y}
The following statements are equivalent for an element $a$ in $R$:
\begin{enumerate}[(1)]
\item[\emph{(1)}] $a$ is morphic.
\item[\emph{(2)}] There exists $b\in R$ such that $aR = \text{Ann}_{R}(b)$ and $\text{Ann}_{R}(a) = bR.$
\item[\emph{(3)}] There exists $b\in R$ such that the sequence below is exact $$ {\cdots}\stackrel{ b\cdot}{\longrightarrow}R_R\stackrel{a\cdot}{\longrightarrow}R_R\stackrel{b\cdot }{\longrightarrow}R_R\stackrel{a\cdot }{\longrightarrow}R_R \stackrel{b\cdot}{\longrightarrow}{\cdots}.$$
\end{enumerate}
\end{corollary}

\begin{corollary}~\label{hth}
Let $R$ be a commutative ring and $M$ be a nontrivial $R$-module.
\begin{enumerate}[(1)]
\item[\emph{(1)}] If $M$ is Artinian and $\text{Ann}_{M}(a) = 0, a\in R$,  then $M$ is $a$-morphic.
\item[\emph{(2)}] If ${}M$ is Noetherian and $Ma = M, a\in R$,  then $M$ is $a$-morphic.
\end{enumerate}
\end{corollary}
\begin{proof}
 For any $a\in R,  \text{Ann}_{M}(a) = 0$ implies $\varphi_a\in S$ is a monomorphism, and $Ma=M$ implies $\varphi_a$ is an epimorphism.  By  \cite[Lemma 11.6]{anderson1992rings}, each case implies that  $\varphi_a$ is an automorphism.   Thus $\varphi_a$ is  morphic.
\end{proof}

\begin{proposition}~\label{efr} Let $a\in R$ and  $M$ be an $a$-morphic $R$-module.  The following statements are equivalent:
\begin{enumerate}[(1)]
\item[\emph{(1)}] $\text{Ann}_{M}(a)=0$;
\item[\emph{(2)}] $Ma = M$;
\item[\emph{(3)}] $\varphi_a\in S$  is an automorphism of $M$.
\end{enumerate}
\end{proposition}
\begin{proof}~
(1)$\Leftrightarrow$(2) Since, by hypothesis, $M/Ma\cong \text{Ann}_{M}(a)$, we have $Ma = M$ if and only if $\text{Ann}_{M}(a)= 0$.

(2)$\Leftrightarrow$(3)
Assume that (2) holds.  Then  $0\cong M/Ma\cong \text{Ann}_{M}(a)$ because $M$ is $a$-morphic.   This gives $\ker(\varphi_a)=0,\varphi_a$ is injective and thus an automorphism of $M$.  Conversely,  the   surjectivity of  $\varphi_a$ gives  $Ma=\varphi_a(M)=M$.
\end{proof}

\noindent The hypothesis `$a$-morphic' cannot be dropped in Proposition~\ref{efr} as seen in Example~\ref{deted}.
\begin{Ex}~\label{deted}
Since $\text{Ann}_{\Bbb Z}(a)=0$ for all $0\neq a\in \Bbb Z$ and  $\varphi_a\notin \text{Aut}_{\Bbb Z}(\Bbb Z)$  when $a\neq\pm 1$,  by Proposition~\ref{efr},  $\mathbb{Z}_{\mathbb{Z}}$ is  never $a$-morphic for any $0,\pm 1\neq a\in \Bbb Z$.
By a similar argument, the free $\Bbb Z$-module $\Bbb Z^n$, for any integer $n>0$, is not $a$-morphic for any $0,\pm 1\neq a\in \Bbb Z$.
For the Pr$\ddot{\text{u}}$fer $p$-group $\Bbb Z(p^\infty)$ ($p$ a prime),  any non-zero endomorphism  is surjective but not injective, in particular also $\varphi_a$.   Thus $\Bbb Z(p^\infty)$  is  not $a$-morphic as a $\Bbb Z$-module for any $\pm 1\neq a\in \Bbb Z$.
\end{Ex}

\noindent Proposition~\ref{efr}  retrieves \cite[Proposition 6]{nicholson2004rings} in Corollary~\ref{efs} for commutative rings by taking $M=R$.

\begin{corollary}~\label{efs} The following statements are equivalent for a morphic element  $a$ of $R$:
\begin{enumerate}[(1)]
\item[\emph{(1)}] $\text{Ann}_{R}(a)=0$;
\item[\emph{(2)}] $aR = R$;
\item[\emph{(3)}] $a\in U(R)$.
\end{enumerate}
\end{corollary}

\begin{lemma}~\label{trace} Let $R$ be an Artinian ring and  $M$ be a nontrivial $R$-module.
\begin{enumerate}[(1)]
\item[\emph{(1)}] If $\text{Ann}_M(a)=0,  a\in R$, then $M$ is $a$-morphic. 
\item[\emph{(2)}] If $Ma=M, a\in R$, then $M$ is $a$-morphic.
\end{enumerate}
\end{lemma}

\begin{proof}~
Let $R$ be an Artinian ring and $a\in R$.  Then the descending chain $aR\supseteq a^2R\supseteq\cdots\supseteq a^nR\supseteq a^{n+1}R\supseteq\cdots$  stabilises.  Since $Ma^n=Ma^nR$ and $\text{Ann}_{M}(a^n)=\text{Ann}_{M}(a^nR)$  for every $n\in \Bbb Z^{+}$, both the  decreasing and increasing sequences  $Ma\supseteq Ma^2\supseteq\cdots\supseteq Ma^n\supseteq Ma^{n+1}\supseteq\cdots$ and  $\text{Ann}_{M}(a)\subseteq \text{Ann}_{M}(a^2)\subseteq\cdots\subseteq \text{Ann}_{M}(a^n)\subseteq \text{Ann}_{M}(a^{n+1})\subseteq\cdots$, respectively,  terminate as well.
\begin{itemize}
\item [$(1)$]
  Suppose that $\text{Ann}_{M}(a)=0=\ker(\varphi_a)$.    If $n\geq1$ is such that $Ma^n=Ma^{n+1}$, then $Ma^t=Ma^{t+1}$ for every $t\geq n$ and so    $Ma^n=Ma^{2n}$.    If $x\in M$, then $xa^n\in Ma^n=Ma^{2n}$  so that  $xa^n=ya^{2n}$ for some $y\in M$.   This implies $(x-ya^n)a^n=0$ so that $x-ya^n\in \text{Ann}_{M}(a^n)=0$.  Now $x=ya^n$ gives $M=Ma^n$.   Since $M=Ma^n$ implies that $M=Ma, \varphi_a$ is an automorphism of $M$.  Hence $M/Ma\cong0= \text{Ann}_{M}(a)$.
\item [$(2)$] Suppose that $Ma=\varphi_a(M)=M$.   If $n\geq1$ is such that $\text{Ann}_{M}(a^n)=\text{Ann}_{M}(a^{n+1})$, then $\text{Ann}_{M}(a^{t})=\text{Ann}_{M}(a^{t+1})$ for every $t\geq n$ so that   $\text{Ann}_{M}(a^n)=\text{Ann}_{M}(a^{2n})$.   Let $x \in \text{Ann}_{M}(a^n)=M\cap \text{Ann}_{M}(a^n)=Ma\cap \text{Ann}_{M}(a^n)=Ma^n\cap \text{Ann}_{M}(a^n)$.  Then $x=ya^n$ for some $y\in M$ and  $0 =xa^n=ya^{2n}$.
     This implies that $y\in \text{Ann}_{M}(a^{2n})=\text{Ann}_{M}(a^n)$; from which we get $x=ya^n= 0$.
Thus  $\text{Ann}_{M}(a^n)=0$.   Since  $\text{Ann}_{M}(a)\subseteq\text{Ann}_{M}(a^n)=0, \text{Ann}_{M}(a)=0$.  This gives $M/Ma\cong0=\text{Ann}_{M}(a)$.
\end{itemize}
\end{proof}

\begin{remark}~\label{has}%
\begin{enumerate}[(a)]
\item[(a)] Lemma~\ref{x} (3) shows that being $a$-morphic (resp., weakly-morphic) is a categorical property, that is, it can be expressed entirely in terms of objects and morphisms.  Therefore, if $R$ and $S$ are rings and $\mathcal{F}:\text{\bf Mod-}R \to \text{\bf Mod-}S$  is a category
equivalence, then an $R$-module $M$ is $a$-morphic (resp. weakly-morphic) if and only if so is the $S$-module $\mathcal{F}(M)$.
\item[(b)] Let $R$ be a ring, $M$ and $N$ be $R$-modules,  $a\in R$ and $\varphi_a\in S$.
The induced functor $\varphi_a^\ast :=\operatorname{Hom}_R(\varphi_a,N): \operatorname{Hom}_R(M, N) \to \operatorname{Hom}_R(M, N),  \psi\mapsto \psi \varphi_a$ for all $\psi:M\to N$ is also multiplication by $a$ (see \cite[pg. 67]{rotman2008introduction}).
 If  $\varphi_a$ is injective,  then  $\varphi_a^\ast :\text{Hom}_R(M,E)\to\text{Hom}_R(M,E)$ is epic for any injective  $R$-module  $E$.
\end{enumerate}
\end{remark}

\noindent An $R$-module $C$ is called a {\it cogenerator} of {\bf Mod}-$R$ if, for every nontrivial $R$-module $M$ and every non-zero $m\in M$, there exists an $R$-homomorphism  $\phi: M\to C$ with $\phi(m)\neq 0$ (see \cite[pg. 264]{rotman2008introduction}).    An $R$-module $E$ is called an {\it injective cogenerator} of {\bf Mod}-$R$ if $E$ is both an injective $R$-module and a cogenerator of  {\bf Mod}-$R$.
It is well known that for any ring $R$,  {\bf Mod}-$R$  has an injective cogenerator (see \cite[Lemma 5.49]{rotman2008introduction}).   Under some special conditions, we now characterise the Artinian-Noetherian rings in terms of $a$-morphic modules.

\begin{proposition}~\label{requ}
Let $R$ be a Noetherian ring and $M$ an $R$-module.  Then the following statements are equivalent:
\begin{enumerate}[(1)]
\item[\emph{(1)}]$R$ is  Artinian.
\item[\emph{(2)}] If $\text{Ann}_{M}(a)=0, a\in R$, then $M$  is $a$-morphic.
\item[\emph{(3)}] If  $M=Ma, a\in R$, then $M$  is $a$-morphic.
\end{enumerate}
\end{proposition}
\begin{proof}~
\begin{itemize}
\item[(1)$\Rightarrow$(3)] Follows from Lemma~\ref{trace}.
\item[(3)$\Rightarrow$(2)]  Let  $E$ be the injective cogenerator of {\bf Mod}-$R, M\in {\bf Mod}\text{-}R$  and   $\varphi_a\in S$.   If $\text{Ann}_{M}(a)=0$, then, by Remark~\ref{has} (b), the functor $$\varphi_a^\ast:\text{Hom}_R(M,E)\to \text{Hom}_R(M,E)$$ is right multiplication by $a$ and epic.   Since, by (3), $\text{Hom}_R(M,E)$ is $a$-morphic as an $R$-module, $\varphi_a^\ast$ is an automorphism by Proposition~\ref{efr}.   Hence $\varphi_a$ is also an automorphism of $M$  since, by \cite[Definition 3 and Theorem 3.1]{ishikawa1964faithfully},  $\varphi_a^\ast$ is   faithfully exact.  Thus (2) follows from Lemma~\ref{x}.
\item[(2)$\Rightarrow$(1)] In view of \cite[Theorem 1]{cohen1950commutative}, it is enough to prove that every prime ideal of  $R$ is maximal.
Let $P$ be any prime ideal of $R$ which is not maximal.   Then there  exists a proper ideal $Q$ of $R$ which is strictly greater than $P$ and an element $a$ in $Q$ but not in $P$ such that,  by \cite[pg. 900]{vasconcelos1970injective}, the induced endomorphism $\varphi_a$ of the $R$-module $M$:$=R/P$ is a monomorphism which is not an epimorphism.  So $\text{Ann}_M(a)=0$.    Since  $M$ is $a$-morphic  by (2),  $\varphi_a$ is an automorphism of $M$ by Proposition~\ref{efr}, which is a contradiction.
\end{itemize}
\end{proof}

\begin{remark}
As an application of Proposition~\ref{requ}, note that $\Bbb Z_{\Bbb Z}$ is a module over  a commutative and  Noetherian ring  with $\text{Ann}_{\Bbb Z}(a)=0$ for each $a\in \Bbb Z$.   However,  $\Bbb Z$ is not Artinian since $\Bbb Z_{\Bbb Z}$ is not $a$-morphic for each $a\in \Bbb Z$ by Example~\ref{deted}.    For the Pr$\ddot{\text{u}}$fer $p$-group $M:=\Bbb Z(p^\infty)$ ($p$ a prime),  $M=Ma$ for each $a\in \Bbb Z$, but $M$  is not $a$-morphic for $a\neq0,\pm1$  by Example~\ref{deted}.  Hence $\Bbb Z$ is not Artinian.
\end{remark}
\noindent A global version of Proposition~\ref{requ} for integral domains is given in Section 4.   The following Proposition~\ref{51} gives further characterisation of weakly-morphic modules.

\begin{proposition}~\label{51}
Let $R$ be a ring and $M$  an $R$-module. The following statements are equivalent.
\begin{enumerate}[(1)]
\item[\emph{(1)}] $M$ is  weakly-morphic.
\item[\emph{(2)}] $M_D$ is a weakly-morphic $R_D$-module for each multiplicative set $D\subseteq R$.
\item[\emph{(3)}] Each element of $R/\text{Ann}_R(M)$ is  morphic as an endomorphism  of $M$.
\item[\emph{(4)}] $M$ is weakly-morphic as an $R/A$-module for  every ideal $A\subseteq\text{Ann}_R(M)$.
\end{enumerate}
\end{proposition}

\begin{proof}~
\begin{itemize}
\item [(1)$\Rightarrow$(2)] Let $M_R$ be a weakly-morphic module, {\bf Mod-}$R$  be a category of all $R$-modules, $D\subseteq R$ be a multiplicative set of $R$, and  $R_D$ and $M_D$ be the localisation of $R$ and $M$ respectively at $D$.  Since   the localisation functor $${\bf Mod}\text{-}R\to {\bf Mod}\text{-}R_D, M\mapsto M_D$$ is  exact by \cite[Corollary 4.81]{rotman2008introduction},  $M_D$ is   weakly-morphic  by   Remark~\ref{has} (a).
\item [(2)$\Rightarrow$(1)] This follows immediately by taking a multiplicative subset  $D=\{1\}$ of $R$.
\item [(1)$\Leftrightarrow$(3)] This follows from (1) of Lemma~\ref{x} and \cite[Lemma 1]{nicholson2005morphic}.
\item [(3)$\Leftrightarrow$(4)] Since $M(a+A)=Ma$ and $\text{Ann}_M(a+A)=\text{Ann}_M(a)$ for every $A\subseteq\text{Ann}_R(M)$ and $a\in R$, we have $M/M(a+A)\cong \text{Ann}_M(a+A)$ for every $a+A\in R/A$ if and only if $M/Ma\cong \text{Ann}_M(a)$ for every $a\in R$.
\end{itemize}
\end{proof}

\noindent Weakly-morphic modules are not closed under submodules.  The module $\mathbb{Q}_{\mathbb{Z}}$  is morphic,   but the  submodule $\mathbb{Z}$ of $\mathbb{Q}$ is not weakly-morphic by Example~\ref{deted}.    Unlike for morphic modules (see  \cite[Example 12]{nicholson2005morphic}),  direct sums and direct products of weakly-morphic modules are  weakly-morphic.

\begin{proposition}~\label{p}
Let $\{M_i\}_{i\in I}$  be any nonempty family  of $R$-modules.  If  $M_i$ is a weakly-morphic $R$-module for all $i\in I$, then
\begin{enumerate}
\item[\emph{(1)}] $\prod_{i\in I} M_i$ is a weakly-morphic $R$-module;
\item[\emph{(2)}] $\bigoplus_{i\in I} M_i$ is a weakly-morphic $R$-module.
\end{enumerate}
\end{proposition}

\begin{proof}~
\begin{itemize}
\item[$(1)$]Since for each $i\in I$,  $M_i/M_ia\cong\text{Ann}_{M_i}(a)$ for every $a\in R$, it follows that
\begin{equation*}
\prod_{i\in I} M_i\Bigg /\left(\prod_{i\in I} M_i \right)a = \prod_{i\in I} M_i \Bigg/\prod_{i\in I} (M_ia) \cong \prod_{i\in I} \left(M_i/M_ia\right)
\cong \prod_{i\in I}\text{Ann}_{M_i}(a).
\end{equation*}
Hence
\begin{equation*}
\prod_{i\in I} M_i \Bigg/\left(\prod_{i\in I} M_i\right)a\cong \text{Ann}_{\prod_{i\in I} M_i}(a).
\end{equation*}

\item[$(2)$] Assume that for each $i\in I,M_i/M_ia\cong\text{Ann}_{M_i}(a)$ for every $a\in R$.  Then
\begin{equation*}
\bigoplus_{i\in I} M_i \Bigg/\left(\bigoplus_{i\in I} M_i\right)a=\bigoplus_{i\in I} M_i \Bigg/\bigoplus_{i\in I} (M_ia)\cong\bigoplus_{i\in I} (M_i/M_ia)\cong \bigoplus_{i\in I}\text{Ann}_{M_i}(a).
\end{equation*}
Thus
\begin{equation*}
\bigoplus_{i\in I} M_i \Bigg/\left(\bigoplus_{i\in I} M_i\right)a=\text{Ann}_{\bigoplus_{i\in I} M_i}(a).
\end{equation*}
\end{itemize}
\end{proof}

\begin{question}~\label{thn}
Is a direct summand of a weakly-morphic module weakly-morphic?
\end{question}
\noindent   Much as Proposition~\ref{p} is false for morphic modules, it was shown to hold if the condition $\text{Hom}(M_i ,M_j)=0, i,j\in I$ with $i\neq j$
in \cite[Lemma 25]{nicholson2005morphic} applies, and the converse   holds without any conditions (see \cite[Theorem 23]{nicholson2005morphic});  that is, direct  summands  of   morphic  modules  are  again  morphic. For a partial answer to Question~\ref{thn},  we have Proposition~\ref{g}.  First we make the following observation in Remark~\ref{ther}.

\begin{remark}~\label{ther}
For any family of endomorphisms $\{f_i:M_i\to M_i\}_{i\in I}$, there is an endomorphism $f:\bigoplus_{i\in I}M_i\to \bigoplus_{i\in I}M_i$
given by $f((m_i)_{i\in I})=(f_i(m_i))_{i\in I}$  with $m_i\in M_i$ (see \cite[pg. 68]{rotman2008introduction}).   In this case \cite{blyth2018module},
\begin{equation*}
\ker(f)=\bigoplus_{i\in I} \ker(f_i)~\text{and}~\text{Im}(f)=\bigoplus_{i\in I}\text{Im}(f_i).
\end{equation*}
Moreover, if  $\{g_i:M_i\to M_i\}_{i\in I}$ is any other family of endomorphisms $g=(g_i)_{i\in I}$,  then
\begin{equation*}
\bigoplus_{i\in I}M_i\stackrel{g}{\longrightarrow}\bigoplus_{i\in I}M_i\stackrel{f}{\longrightarrow}\bigoplus_{i\in I}M_i
\end{equation*}
is an exact sequence if and only if for every $i\in I$, the sequence below is  exact (\cite{blyth2018module})
\begin{equation*}
 M_i\stackrel{g_i}{\longrightarrow}M_i\stackrel{f_i}{\longrightarrow}M_i.
\end{equation*}
If $\varphi_a=(\varphi_{i_a})_{i\in I}$ is a right multiplication by $a$ endomorphism of $M=\bigoplus_{i\in I}M_i$
 defined by $\varphi_a(m)=ma$,  then   $\varphi_{i_a}(m_i)=m_ia$ for each $i\in I$.
\end{remark}
\begin{proposition}~\label{g}
If $\{M_i\}_{i\in I}$ is a family of $R$-modules satisfying  $\text{Hom}(M_i ,M_j)=0$ for every $i\neq j$ and  $M=\bigoplus_{i\in I} M_i$,  then $M$ is weakly-morphic if and only if for every $i\in I$  all the modules $M_i$ are weakly-morphic.
\end{proposition}
\begin{proof}~
\begin{itemize}
\item[($\Rightarrow$):] Assume that $M=\displaystyle{\bigoplus_{i\in I} M_i}$ is a weakly-morphic module.     Since $\text{Hom}(M_i,M_j)=0$ for every $i,j\in I$ with $i\neq j$,    $\text{End}(M) \cong \prod_{i\in I}\text{End}(M_i)$ by \cite[Lemma 1.7]{anderson2021endoregular}.   For any $\psi=(\psi_i)_{i\in I}\in\text{End}(M)$ and $m=(m_i)_{i\in I}\in M$ with $m_i\in M_i,i\in I$, we obtain   $\psi(m)=\psi((m_i)_{i\in I})=(\psi_i(m_i))_{i\in I}\in
\bigoplus_{i\in I}\psi_i(M_i)$, where this sum is well defined because almost all $m_i$ vanish.     It follows that  $\text{Im}(\psi)=\bigoplus_{i\in I}\text{Im}(\psi_i)$ and $\ker(\psi)=\bigoplus_{i\in I}\ker(\psi_i)$.
Now consider any $\varphi_a=(\varphi_{i_a})_{i\in I}\in \text{End}(M)$  such that
  $\bigoplus_{i\in I}\varphi_{i_a}(M_i)= \varphi_a(M)=ker(\psi)=\bigoplus_{i\in I}\ker(\psi_i)$
   and $\bigoplus_{i\in I}\psi_i(M_i)= \psi(M)=ker(\varphi_a)=\bigoplus_{i\in I}\ker(\varphi_{i_a})$  for some $\psi=(\psi_i)_{i\in I}\in \text{End}(M)$.   By the preceding discussion in Remark~\ref{ther},  the lower row of the commutative diagram
 $$\begin{CD}
{\cdots}@>\psi_{i}>>M_i @>\varphi_{i_a}>> M_i @>\psi_{i}>> M_i@>\varphi_{i_a}>> M_i@>\psi_{i}>> {\cdots}\\
@V~VV@V\iota_1VV @V\iota_2VV @V\iota_3VV@V\iota_4VV@V~VV \\
{\cdots}@>\psi >>\displaystyle{\bigoplus_{i\in I} M_i} @> \varphi_a>>\displaystyle{\bigoplus_{i\in I} M_i} @>\psi >> \displaystyle{\bigoplus_{i\in I} M_i} @> \varphi_a>>\displaystyle{\bigoplus_{i\in I} M_i}@>\psi >>{\cdots}
\end{CD}
$$
with canonical injective maps $\iota_1,\iota_2,\iota_3,\iota_4$ is exact if and only if the upper row is exact.  Since  $\varphi_{i_a}(m_i)=m_ia$ for each $i\in I$, each $M_i$ is a weakly-morphic module by  Lemma~\ref{x}.
\item[($\Leftarrow$):] The converse of the proposition follows from Proposition~\ref{p}.
\end{itemize}
\end{proof}

\begin{Ex}~\label{gtg}
 Weakly-morphic modules are not closed under quotients.  Using the $\Bbb Z$-module  $\mathbb{Q}$ and its submodule $\mathbb{Z}$,  the quotient  $\Bbb Q/\Bbb Z=\bigoplus_{p\in \mathcal{P}}\Bbb Z(p^\infty)$ satisfies the property $\text{Hom}_{\Bbb Z}(\Bbb Z(p^\infty),\Bbb Z(q^\infty))=0,p\neq q$ for any  $p,q\in \mathcal{P}$ where $\mathcal{P} \subseteq\Bbb Z$ is a  collection of primes.
  Since the $(p)$-primary component  $\Bbb Z(p^\infty)$ of $\Bbb Q/\Bbb Z$ is  not weakly-morphic as a $\Bbb Z$-module by Example~\ref{deted},
  $\displaystyle{\Bbb Q/\Bbb Z}$  is not   weakly-morphic as $\Bbb Z$-module by Proposition~\ref{g}.
\end{Ex}

\begin{corollary}
Let $R$ be a  ring and  $\{M_i\}_{i\in I}$ be a family of $R$-modules.   If   $M=\bigoplus_{i\in I} M_i$ is a  weakly-morphic module with a reduced ring of endomorphisms, then  for every $i\in I$,  all the modules $M_i$ are weakly-morphic.
\end{corollary}
\begin{proof}
Since the  ring of endomorphisms $\text{End}_R(M)$ of $M=\bigoplus_{i\in I} M_i$  is reduced, every idempotent endomorphism of $M$ is central.   By \cite[pg. 536]{ozcan2006duo}, every direct
summand of $M$ is fully invariant.    Thus  $\text{Hom}(M_i ,M_j)=0$ for every $i\neq j$.   Applying Proposition~\ref{g}, each $M_i,i\in I$ is a weakly-morphic module.
\end{proof}

\noindent Since weakly-morphic and morphic are indistinguishable for rings, by \cite[Example 2]{nicholson2004rings}, a direct product $R:=\prod_{i\in I}R_i$ of rings $R_i$ is  (weakly-)morphic if and only if each $R_i$ is (weakly-)morphic.

\begin{proposition}~\label{gtg}
Suppose that $M_i$ is an $R_i$-module for each $i=1,\ldots,n$. Then $\prod_{i=1}^nM_i$ is a weakly-morphic $\prod_{i=1}^nR_i$-module if and only if $M_i$ is a weakly-morphic $R_i$-module for each $i=1,\ldots,n$.
\end{proposition}
\begin{proof}~
\begin{itemize}
\item[($\Rightarrow$):] Assume that $M_1\times M_2$ is a weakly-morphic $(R_1\times R_2)$-module for the case $i=1,2$.  Take $a_1\in R_1$ and define $\vartheta_{a_1}:M_1\to M_1$ by $\vartheta_{a_1}(m_1)=m_1a_1$ for every $m_1\in M_1$.   Then $\vartheta_{a_1}$ is a right multiplication by $a_1$ endomorphism of $M_1$.
  Further, defining a map   $\vartheta_{(a_1,1)}:M_1\times M_2\to M_1\times M_2$ by $\vartheta_{(a_1,1)}((m_1,m_2))=(m_1a_1,m_2)=(m_1,m_2)(a_1,1)$ for every $(m_1,m_2)\in M_1\times M_2$ gives a right multiplication by $(a_1,1)$ endomorphism of $M_1\times M_2$.   Then  $\vartheta_{(a_1,1)}(M_1\times M_2)=M_1a_1\times M_2$ and $\ker(\vartheta_{(a_1,1)})\cong\text{Ann}_{M_1}(a_1)$.   Since  $M_1\times M_2$ is weakly-morphic,
\begin{eqnarray*}
  M_1/M_1a_1 &\cong& \frac{M_1}{M_1a_1} \times \frac{M_2}{M_2} \cong \frac{M_1\times M_2}{M_1a_1\times M_2}\\
  &=& \frac{M_1\times M_2}{\vartheta_{(a_1,1)}(M_1\times M_2)}  \cong  \ker(\vartheta_{(a_1,1)})\\
  &\cong&\text{Ann}_{M_1}(a_1).
\end{eqnarray*}
This proves that   $M_1$ is a weakly-morphic  $R_1$-module.   Similarly, $M_2/M_2a_2\cong \text{Ann}_{M_2}(a_2)$ for each $a_2\in R_2$ so that $M_2$ is a weakly-morphic $R_2$-module.   Consequently,  by induction on $n, M_i$ is a weakly-morphic $R_i$-module for each $i=1,\ldots,n$.
\item[]$(\Leftarrow)$:  Assume that for each $i=1,\ldots,n, M_i$ is a weakly-morphic $R_i$-module.     Then  $M_i/M_ia_i\cong \text{Ann}_{M_i}(a_i)$ for each $a_i\in R_i$ with $i=1,\ldots,n$.
        Now put $a=(a_1,\ldots,a_n)\in \prod_{i=1}^nR_i$.
        Then we have
\begin{eqnarray*}
\prod_{i=1}^nM_i \Bigg/\left(\prod_{i=1}^nM_i\right)(a_1,\ldots,a_n)&=&\prod_{i=1}^nM_i\Bigg/\prod_{i=1}^n(M_ia_i)\\
&\cong& \prod_{i=1}^n(M_i/ M_ia_i)\\
&\cong&  \prod_{i=1}^n\text{Ann}_{M_i}(a_i).
\end{eqnarray*}

     This implies that  $\displaystyle{\prod_{i=1}^nM_i \Bigg/\left(\prod_{i=1}^nM_i\right)a\cong \text{Ann}_{\prod_{i=1}^nM_i}(a)}$   which completes the proof.
\end{itemize}
\end{proof}

\section{Modules with regular endomorphisms}

\noindent  We say  $a\in R$ is  {\it regular} (resp., {\it strongly regular}) if there exists some $x\in R$ such that $a = axa$ (resp., $a = a^2x$).
   A ring is regular (resp., strongly regular) if every $a\in R$ is regular (resp., strongly regular). If $a\in R$ is strongly regular, then $a = ue$ where $u$ is a unit and $e$ is an
idempotent.
It is well known by \cite[Lemma 3.1]{ware1971endomorphism} that $\varphi\in S$  is regular  if  and  only  if  both $\varphi(M)$  and  $\ker(\varphi)$  are  direct summands  of  $M$.

\begin{proposition}\label{e5e}
Let $R$ be a ring, $M$ an $R$-module and consider the following statements:
\begin{enumerate}
\item[\emph{(1)}] $\varphi_a\in S$  is  regular for every $a\in R$;
\item[\emph{(2)}] $M=Ma\bigoplus \text{Ann}_M(a)$ for every $a\in R$;
\item[\emph{(3)}] $Ma$ and $\text{Ann}_M(a)$ are direct summands of $M$ for every $a\in R$;
\item[\emph{(4)}] $M$ is a weakly-morphic module.
  \end{enumerate}
Then we always have the following implications: $(1)\Leftrightarrow(2)\Leftrightarrow(3)\Rightarrow(4)$.
\end{proposition}
\begin{proof}~
\begin{itemize}
  \item [(1)$\Rightarrow$(2)]
Since $\varphi_a$ is a commutative element of $S$ and regular, there exists some $\psi\in S$ such that $\varphi_a=\varphi_a\psi\varphi_a=\varphi_a^2\psi$.  So $\varphi_a$ is strongly regular, that is,  $\varphi_a=e\beta$ where $e$ is a central idempotent endomorphism of $M$ and $\beta$ is some automorphism of $M$.   Then $\varphi_a(M) = e(M)$ and $\ker(\varphi_a) = \ker(e) =(1_M-e)(M)$.   Hence $M = \varphi_a(M)\bigoplus\ker(\varphi_a)$.  Now (2) follows.
\item [(2)$\Rightarrow$(3)] and $(2)\Rightarrow(4)$ are clear.  $(3)\Leftrightarrow(1)$ follows from \cite[Lemma 3.1]{ware1971endomorphism}.
\end{itemize}
\end{proof}
\noindent
$(4)\Rightarrow(3)$ in Proposition~\ref{e5e} fails in general.  The $\Bbb Z$-module $M:=\Bbb Z_4$  is weakly-morphic but for $a:=\overline{2}$, $Ma\bigoplus\text{Ann}_M(a)=\overline{2}\Bbb Z_4\bigoplus\overline{2}\Bbb Z_4\cong \Bbb Z_2\bigoplus\Bbb Z_2\neq M$.

\begin{corollary}~\label{ftft}
Every module over a regular ring is weakly-morphic.
\end{corollary}
\begin{proof}
 If $M$ is a module over a regular ring $R$, then $R/\text{Ann}_R(M)$ is regular.  Since $R/\text{Ann}_R(M)\cong S_M$, $M$ is weakly-morphic  by Proposition~\ref{e5e}.
\end{proof}

\noindent A module $M$ is  {\it kernel-direct} (resp., {\it image-direct}) if   $\ker(\varphi)\subseteq^{\bigoplus}M$ (resp., $\varphi(M)\subseteq^{\bigoplus}M$) for every
$\varphi\in S$  \cite{nicholson2005morphic}.    Kernel-direct  (resp., image-direct) modules are also called Rickart (resp., Dual-Rickart) modules (see  \cite{lee2010rickart} and \cite{lee2011rickart}).

\begin{theorem}\label{datadata} If an $R$-module $M$ is either kernel-direct or image-direct, then all the four statements in Proposition~\ref{e5e} are equivalent.
\end{theorem}
\begin{proof} It is enough to prove that $(4)\Rightarrow(2)$  in each case.   The rest of the implications have been proved in Proposition~\ref{e5e}.
\begin{itemize}
\item[Case 1:] Assume that  $M$ is kernel-direct and weakly-morphic.   Then $\text{Ann}_M(a)=\ker(\varphi_a)\subseteq^{\bigoplus}M$ for each $a\in R$.  Further, by Lemma~\ref{x}, for each $a\in R$ there exists $\psi\in S$ such that $Ma=\ker(\psi) \subseteq^{\bigoplus}M$.    It follows from Proposition~\ref{e5e} that $M=Ma\bigoplus \text{Ann}_M(a)$.
\item [Case 2:] Assume that  $M$ is image-direct and weakly-morphic.   Then for each  $a\in R,Ma= \varphi_a(M) \subseteq^{\bigoplus}M$.   By Lemma~\ref{x},  $\psi\in S$ exists such that  $\text{Ann}_M(a)=\ker(\varphi_a)=\psi(M)\subseteq^{\bigoplus}M$.   Hence $M=Ma\bigoplus \text{Ann}_M(a)$ by Proposition~\ref{e5e}.
\end{itemize}
\end{proof}

\begin{corollary}~\label{}
A kernel-direct or image-direct $R$-module is weakly-morphic if and only if each element of  $R/\text{Ann}_R(M)$ is regular as an element of $S$, that is, for each $a\in R$ there exists  $\varphi\in S$  such that $ma= \varphi(m)a^2$ for each $m\in M$.
\end{corollary}
\begin{proof}
Since $S_M\cong R/\text{Ann}_R(M)$,  Proposition~\ref{e5e} and Theorem~\ref{datadata} complete the proof of the corollary.
\end{proof}

\noindent An $R$-module $M$ is {\it endoregular}  if $S$ is a regular ring, equivalently, if $\varphi(M)\subseteq^{\bigoplus} M$ and $ker(\varphi)\subseteq^{\bigoplus}M$ for every $\varphi\in\text{End}_R(M)$  \cite{lee2013modules} and \cite[Corollary 3.2]{ware1971endomorphism}.
\begin{corollary}~\label{yayeri}
Every endoregular module is a weakly-morphic module.
\end{corollary}
\begin{proof}
An endoregular module $M$ is both kernel-direct and image-direct.   Therefore,  $Ma$ and $\text{Ann}_M(a)$ are direct summands of $M$.    Hence $M$ is weakly-morphic by Proposition~\ref{e5e}.
\end{proof}

\noindent  A module $M$ is {\it Dedekind-finite} if  it is not isomorphic to any proper direct summand of itself, \cite[pg. 185]{Calugareanu2010abelian}.  Every morphic module is Dedekind-finite module (see~\cite{nicholson2005morphic}).\\

\begin{Ex}
 An endoregular module need not be morphic.   The  $\Bbb Z$-module $\bigoplus_{i=1}^\infty\Bbb Q_i$, where $\Bbb Q_i=\Bbb Q$,   is endoregular but not morphic since it is not Dedekind-finite.   Also the  endoregular module $\bigoplus_{i=1}^\infty\Bbb R_i$,  where $\Bbb R_i=\Bbb R$, is not morphic by \cite[Corollary to Theorem 1]{ehrlich1976units}.
Moreover, the  $R$-module  $\bigoplus_{n=1}^\infty R$ over a non-Noetherian regular ring $R$ is weakly-morphic by  Corollary~\ref{ftft} but  not endoregular by \cite[Theorem 2]{shanny1971regular}.
\end{Ex}

\noindent    A submodule  $N\subseteq M$  is  {\it relatively divisible-pure} or {\it RD-pure}  in  $M$  in case  $Na=Ma\cap N$ (equivalently, $0\to N\bigotimes R/aR\to M\bigotimes R/aR$ is exact)  for each $a\in R$ \cite[Proposition 2]{warfield1969purity}.
 Proposition~\ref{dadada} gives another class of  weakly-morphic modules.

\begin{proposition}~\label{dadada}
Let $R$ be a ring and $M$ an $R$-module in which every submodule is RD-pure.  Then $M$ is a weakly-morphic module.
\end{proposition}
\begin{proof}
We will show that $M=Ma\bigoplus \text{Ann}_M(a)$  for each $a\in R$.   The rest follows from Proposition~\ref{e5e}. Let $\varphi_a\in S$.     Then  $\varphi_a(M)=Ma=Ma\cap Ma=Ma^2=\varphi_a^2(M)$.   So $\varphi_a(M)=\varphi_a^2(M).$    For any   $m\in M$,   $\varphi_a(m)=\varphi_a^2(n)$ for some $n\in M$.   Since $\varphi_a(m-\varphi_a(n))=0, m-\varphi_a(n)\in \ker(\varphi_a)$ and     $m=\varphi_a(n)+m-\varphi_a(n)\in \varphi_a(M)+\ker(\varphi_a)$.
This implies that $M=\varphi_a(M)+\ker(\varphi_a)$.     Next we  show that this is actually a direct sum decomposition.   By the RD-pure property, $0=\text{Ann}_M(a)a=Ma\cap\text{Ann}_M(a)=\varphi_a(M)\cap\ker(\varphi_a)$ for every $a\in R$.  Hence   $M=\varphi_a(M)\bigoplus\ker(\varphi_a)=Ma\bigoplus \text{Ann}_M(a)$.
\end{proof}

\noindent
A submodule $N$ of $M$  is {\it pure} in $M$ if the sequence $0\to N\bigotimes E\to M\bigotimes E$  is exact for each $R$-module $E$.
    By \cite[Proposition 8.1]{Fieldhouse}, every pure submodule is also RD-pure.  Therefore, by Proposition~\ref{dadada},  all modules whose every  submodule is pure are  weakly-morphic.

\begin{Ex}~\label{haga}
\begin{enumerate}
\item[(a)] The converse of Proposition~\ref{dadada} does not hold in general.  The $\Bbb Z$-module $\Bbb Q$  is weakly-morphic but not all its submodules are  (RD-)pure since $2\Bbb Q\cap\Bbb Z\neq 2\Bbb Z$ for the submodule $\Bbb Z$.
\item[(b)] The modules described by Proposition~\ref{dadada} are not necessarily morphic.  If $M$ is a  module  in which every principal submodule is a direct summand (equivalently,
\cite[Corollary 1.3]{zelmanowitz1972regular} every finitely generated submodule is a direct summand),  then every submodule of $M$ is RD-pure by \cite[pg. 240 and 246]{Ramamurth} and  weakly-morphic by Proposition~\ref{dadada}.  However,  by \cite[Corollary 33]{nicholson2005morphic}, $M$ need not be morphic unless $S$ is right morphic.
\end{enumerate}
\end{Ex}

\noindent  An $R$-module $M$ is  a {\it duo module} provided every submodule of $M$ is fully invariant, that is, any  submodule  $N$ of $M$, $\varphi(N)\subseteq N$ for every $\varphi\in S$ \cite{ozcan2006duo}.      A ring $R$ is  {\it right (left) duo ring} if the right (left) $R$-module $R$ is a duo module and {\it duo} if it is both right and left duo.  Clearly, commutative rings  are duo rings.  An $R$-module $M$ is said to be a {\it multiplication module} provided for each submodule $N$ of $M$ there exists an ideal $A$ of $R$ such that $N =MA$ \cite{choi1994endomorphisms}.  Multiplication modules are a source of  duo modules \cite[pg. 534]{ozcan2006duo}.

\begin{proposition}~\label{300}
Let $R$ be a ring and $M$ be a finitely generated multiplication $R$-module.   Then $M$  is weakly-morphic if and only if it is morphic.
\end{proposition}

\begin{proof}
Suppose that $M$ is a  finitely generated multiplication module.  Then  every endomorphism  $\varphi\in S$ is of the form $\varphi_a$ for some $a\in R$ by  \cite[Theorem 3]{choi1994endomorphisms}.   Since   $M$ being weakly-morphic implies  every endomorphism in $S$ is   morphic,  $M$  is a morphic  module.    The converse always holds.
\end{proof}

\begin{remark}~\label{goat}
\begin{enumerate}
\item[(a)] A module $M$ is \emph{uniserial} if its submodule lattice forms a  chain.    The $\Bbb Z$-module $\Bbb Z_{p^\infty}$ is uniserial but  not weakly-morphic by Example~\ref{gtg}.    Since commutative rings are duo, by \cite[Proposition  8]{nicholson2005morphic},  every uniserial $R$-module  of finite length   is (weakly-)morphic.
\item[(b)] Since every cyclic $R$-module is a multiplication module that is finitely generated,  it is weakly-morphic if and only if it is morphic by Proposition~\ref{300}.
\end{enumerate}
\end{remark}

\begin{corollary}
Every endoregular duo module is a morphic module.
\end{corollary}
\begin{proof}
Let $M$ be an endoregular module that is duo.  Then the ring of endomorphisms $S$  of $M$ is regular.  We show that $S$   is unit-regular.
Let $\varphi,\psi\in S$ and $m\in M$.
Since $M$ is duo,  using   \cite[Lemma 1.1]{ozcan2006duo}, $\varphi(m)=ma$ and $\psi(m)=mb$ for some $a$ and $b$ in $R$.  Commutativity of $R$ implies $\varphi(\psi(m)=\varphi(mb)=mba=mab=\psi(ma)=\psi(\varphi(m)).$     Thus  $S$ is a commutative regular ring and hence unit-regular.     Consequently,  $M$ is morphic by \cite[Example  28]{nicholson2005morphic}.
\end{proof}

\section{Weakly-morphic modules over integral domains}
\noindent Let $D$ be an integral domain.  A $D$-module  $M$  is called {\it divisible}  if  $Ma=M$  for each $0\neq a\in D$.  $M$  is said to be {\it torsion-free} if $ma = 0$ implies $a= 0$ or $m = 0$ for each $m\in M$ and $a\in D$.

\begin{proposition}~\label{capsa}
Let  $D$ be a principal ideal domain  and $M:=\bigoplus_{i\in I} M_i$ be a direct sum of finitely generated torsion-free $D$-modules $M_i$.  Then $D$ is  morphic if and only if $M$ is  weakly-morphic.
\end{proposition}

\begin{proof}~
\begin{itemize}
\item[] $(\Rightarrow)$: By the classical theorem for the decomposition of finitely generated modules over principal ideal domains, each $M_i$ is a (finite) direct sum of
cyclic submodules $M_{ij},j=1,2,\ldots k_i$ for some positive integer $k_i$.
Since $M$ is torsion-free, $M_i$ is torsion-free and so is $M_{ij}$.  Further,  for each $j=1,2,\ldots k_i$ $$M_{ij}=mD\cong D/\text{Ann}_{D}(m)$$ for some $0\neq m\in M_{ij}$.   As  $M_{ij}$ is torsion-free, $\text{Ann}_{D}(m)=0$ and $M_{ij}\cong D_{D}$ for each $j=1,2,\ldots k_i$.    If $D$ is  morphic, then
$$M=\displaystyle{\bigoplus_{i\in I} M_i}=\bigoplus_{i\in I}\left(\bigoplus_{j=1}^{k_i}M_{ij}\right)   \cong \bigoplus_{i\in I}\left(\bigoplus_{j=1}^{k_i}D_{ij}\right)$$ with each $D_{ij}=D_{D}$ morphic.   Applying Proposition~\ref{p}, $\displaystyle{\bigoplus_{j=1}^{k_i} D_{ij}}$ is weakly-morphic as a $D$-module, and so is $M$.
\item[($\Leftarrow$):]
Assume that $M$ is a weakly-morphic module.  Since $M$ is torsion-free, $\text{Ann}_M(a)=0$ for every $0\neq a\in D$.   Applying Proposition~\ref{efr},  $M=Ma$.
Since $M$ is a finitely generated module, by \cite[Corollary 2.5]{Atiyah1969}, we have $M(1 + ad)=0$ for some $d\in D$.  This implies that $1+ad\in \text{Ann}_D(M)$.  As $M$ is faithful,    $1=a(-d)$ and  $a$ is a unit element of $D$.  Hence $D$ is a field, which  proves that $D$ is morphic.
\end{itemize}
\end{proof}

\begin{corollary}~\label{caps}
Let $D$ be a principal ideal domain and $M$ be a direct sum of torsion-free cyclic $D$-modules.  Then  $M$ is  weakly-morphic  if and only if $D$ is morphic.
\end{corollary}
\begin{proof}
Since a cyclic module is finitely generated, the result follows  from Proposition~\ref{capsa}.
\end{proof}

\begin{theorem}~\label{fafa}
Let $D$ be an integral domain and  $M$ a weakly-morphic $D$-module.   The following statements are equivalent:
\begin{enumerate}[(1)]
\item[\emph{(1)}] $M$ is torsion-free;
\item[\emph{(2)}] $M$ is divisible;
\item[\emph{(3)}] $S_M$ is a field.
\end{enumerate}
\end{theorem}
\begin{proof}
Since the $D$-module $M$ is torsion-free (resp., divisible) if and only if $\text{Ann}_M(a)=0$ (resp., $Ma=M$) for every $0\neq a\in D$, the equivalences $(1)\Leftrightarrow(2)\Leftrightarrow(3)$  follow from Proposition~\ref{efr}.
\end{proof}

\noindent Denote the field  of fractions of an integral domain $D$  by $\text{Frac}(D)$.
\begin{corollary}~\label{oil}
Let $D$ be a principal ideal domain and  $M$ a weakly-morphic $D$-module.   The following statements are equivalent:
\begin{enumerate}[(1)]
\item[\emph{(1)}] $M$ is torsion-free;
\item[\emph{(2)}] $M$ is divisible;
\item[\emph{(3)}] $M$ is injective;
\item[\emph{(4)}] $S_M$ is a field;
\item[\emph{(5)}] $M$ is a vector space over $Q$, where $Q: = \text{Frac}(D)$.
\end{enumerate}
\end{corollary}
\begin{proof}
 The equivalences (2)$\Leftrightarrow$(3)   and (4)$\Leftrightarrow$(5) follow from  \cite[Corollary 3.35 and pg. 97 resp.]{rotman2008introduction} without weakly-morphic hypothesis and   Theorem~\ref{fafa} completes the proof of the other equivalences.
\end{proof}

\noindent Corollary~\ref{oil} tells us that the weakly-morphic torsion-free  $\Bbb Z$-modules are of the form $M\cong \Bbb Q\bigoplus \Bbb Q\bigoplus\cdots\bigoplus \Bbb Q$.  Now we  give a global version of Proposition~\ref{requ} in Theorem~\ref{ftftf}  for domains.
\begin{theorem}~\label{ftftf}
Let $D$ be a Noetherian integral domain.  The following statements are equivalent:
\begin{enumerate}[(1)]
\item[\emph{(1)}] $D$ is Artinian,
\item[\emph{(2)}] every torsion-free $D$-module is weakly-morphic,
\item[\emph{(3)}] every divisible $D$-module is weakly-morphic.
\end{enumerate}
\end{theorem}
\begin{proof}
Since the $D$-module $M$ is torsion-free (resp., divisible) if and only if $\text{Ann}_M(a)=0$ (resp., $Ma=M$) for every $0\neq a\in D$,  (1)$\Leftrightarrow$(2) and (1)$\Leftrightarrow$(3) are immediate from Proposition~\ref{requ} .
\end{proof}

\noindent An endomorphism $\alpha$ of a  group $G$ is {\it normal} \cite{li2010morphic}  if the image $\alpha(G)$ is a normal subgroup in $G$, and  $\alpha$ is morphic if and only if it is both normal and $G/\alpha(G)\cong\ker(\alpha)$.   A group $G$ is {\it morphic} if every endomorphism of $G$ is morphic.  Note that an Abelian group is morphic if and only if it is morphic as a $\Bbb Z$-module.   It is well-known that if $G$ is a finite Abelian group of order  $n$, then  the endomorphism $\varphi_a(g) = ga$ of $G$, where $g\in G$ and $a\in\Bbb Z$,
 is an automorphism if and only if $(a,n) = 1$, where  $(a,n) = 1$ means that $n$  is coprime with $a$.
  We say that a finitely generated Abelian group is {\it weakly-morphic} if it is weakly-morphic as a $\Bbb Z$-module.
 We will give conditions for   weakly-morphic finitely generated  Abelian groups to be morphic.   But first, we recall  a related result about morphic finitely generated Abelian groups from \cite{nicholson2005morphic}.

\begin{theorem}{\rm\cite[Theorem 26]{nicholson2005morphic}}~\label{rats}  A finitely generated Abelian  group is  morphic  if  and  only if it is finite and  each $(p)$-primary  component has the form  $(\Bbb Z_{p^k})^n$ for  some non-negative integers $n$  and $k$.
\end{theorem}

\noindent  For weakly-morphic finitely generated $\Bbb Z$-modules, we have:

\begin{theorem}~\label{das}
A finitely generated $\Bbb Z$-module is weakly-morphic  if and only if it is finite.
\end{theorem}
\begin{proof}~
\begin{itemize}
  \item[] $(\Rightarrow)$: Suppose that $M$  is a weakly-morphic and finitely generated $\Bbb Z$-module.  By the fundamental theorem  of  finitely generated $\Bbb Z$-modules, $M$ is a direct sum of primary and infinite cyclic groups, that is,
$$M\cong\displaystyle{\bigoplus_{(i)}\Bbb Z/p_i^{n_i}\Bbb Z\bigoplus\Bbb Z^n}$$ where $i,n\in \Bbb Z_{\geq0}$ and $p$ a prime number.  It is enough to show that $M$ contains  no  infinite summand.   Suppose that the summand $\Bbb Z^n$ is nontrivial.  Let $\omega$ be the order of the group $\bigoplus_{(i)}\displaystyle{\Bbb Z/p_i^{n_i}\Bbb Z}$.   For all $0\neq a\in \Bbb Z$ for which $(\omega,a)=1$,  the induced endomorphism $\varphi_a$ is an  automorphism of $\bigoplus_{(i)}\displaystyle{\Bbb Z/p_i^{n_i}\Bbb Z}$ and also injective on $\Bbb Z^n$.   So $\varphi_a$ is injective on $M$ but not
surjective on $\Bbb Z^n$ for any $a\neq\pm1$.
Therefore,  $\varphi_a$ is  not an automorphism of $M$.    In view of Proposition~\ref{efr}, $M$ is not $a$-morphic for all $a\in \Bbb Z$; which creates a contradiction.
Hence  $M$ is a direct sum  of only its  primary  components.  So $M$ is finite.
\item[]$(\Leftarrow)$: Assume that  $M$ is finite.
By the primary decomposition of finite Abelian groups, $$M\cong\bigoplus_{(p)}M_{(p)}$$ where
$M_{(p)}=\Bbb Z/p^{n_1}\Bbb Z\bigoplus \Bbb Z/p^{n_2}\Bbb Z\bigoplus\cdots\bigoplus \Bbb Z/p^{n_k}\Bbb Z$ are the  cyclic $(p)$-primary components and $n_1\geq n_2\geq\cdots\geq n_k$ are non-zero integers.  If $a=p^t$ where $t$ is a non-negative integer, then $M_{(p)}a=M_{(p)}p^t=p^t\Bbb Z/p^{n_1}\Bbb Z\bigoplus p^t\Bbb Z/p^{n_2}\Bbb Z\bigoplus\cdots\bigoplus p^t\Bbb Z/p^{n_k}\Bbb Z.$
It follows  that

\begin{align}
\left.M_{(p)}\right/M_{(p)}a
&=\frac{\Bbb Z/p^{n_1}\Bbb Z\bigoplus \Bbb Z/p^{n_2}\Bbb Z\bigoplus\cdots\bigoplus \Bbb Z/p^{n_k}\Bbb Z}{(p^t\Bbb Z/p^{n_1}\Bbb Z\bigoplus p^t\Bbb Z/p^{n_2}\Bbb Z\bigoplus\cdots\bigoplus p^t\Bbb Z/p^{n_k}\Bbb Z)}\nonumber\\
&\cong \left(\frac{\Bbb Z/p^{n_1}\Bbb Z}{p^t\Bbb Z/p^{n_1}\Bbb Z}\right)\bigoplus\left(\frac{\Bbb Z/p^{n_2}\Bbb Z}{p^t\Bbb Z/p^{n_2}\Bbb Z}\right)\bigoplus\cdots\bigoplus\left(\frac{\Bbb Z/p^{n_k}\Bbb Z}{p^t\Bbb Z/p^{n_k}\Bbb Z}\right)\nonumber\\
&\cong \Bbb Z/p^t\Bbb Z\bigoplus \Bbb Z/p^t\Bbb Z\bigoplus\cdots\bigoplus\Bbb Z/p^t\Bbb Z\nonumber\\
&\cong \text{Ann}_{M_{(p)}}(p^t)\nonumber\\
&= \text{Ann}_{M_{(p)}}(a);\nonumber
\end{align}

and for any $a\neq p^t, \varphi_a$ is an automorphism of $M_{(p)}$.
Thus each $M_{(p)}$ is a weakly-morphic direct summand and, so $M$ is weakly-morphic by Proposition \ref{p}.
\end{itemize}
\end{proof}

\noindent A new characterisation of morphic  finitely generated Abelian groups  in terms of weakly-morphic $\Bbb Z$-modules is the following:
\begin{theorem}~\label{dsa}
A finitely generated Abelian  group is morphic  if and only if it is weakly-morphic as a $\Bbb Z$-module and each $(p)$-primary  component has the form  $(\Bbb Z_{p^k})^n$ for  some  non-negative integers $n$ and  $k$.
\end{theorem}
\begin{proof}
This is immediate from Theorem~\ref{rats} and Theorem~\ref{das}.
\end{proof}
\begin{corollary}
If $D$ is a principal ideal domain, then every finitely generated torsion $D$-module  is weakly-morphic.
\end{corollary}
\begin{proof}
Over a principal ideal domain, a finitely generated torsion $D$-module is a direct sum of cyclic $(p)$-primary modules and also is  finite as an Abelian group.  The proof is complete after applying Theorem~\ref{das}.
\end{proof}

\begin{ack}
We thank Dr. Alex Samuel Bamunoba  for his helpful suggestions and correspondence regarding this material.
 We are also grateful to the anonymous referee for carefully reading the paper and providing valuable corrections and suggestions
that have improved it.  This work was carried out at Makerere University with support from the Makerere-Sida Bilateral Program Phase IV, Project 316 ``Capacity building in mathematics and its application".
\end{ack}

\end{document}